\documentclass[a4paper,12pt]{amsart}

\usepackage{a4wide}
\usepackage{tikz}
\usetikzlibrary{decorations.pathreplacing}
\usetikzlibrary{shapes}

\newif\ifdetails
\detailstrue
\newcommand{\DETAIL}[1]%
{\ifdetails\par\fbox{\begin{minipage}{0.9\linewidth}\textit{Detail:}
      #1\end{minipage}}\par\fi}
\newcommand{\TODO}[1]%
{\ifdetails\par\fbox{\begin{minipage}{0.9\linewidth}\textbf{TODO:}
      #1\end{minipage}}\par\fi}

\newtheorem{lemma}{Lemma}

\newtheorem{theorem}[lemma]{Theorem}
\newtheorem{corollary}[lemma]{Corollary}

\newtheorem{question}{Question}

\newcommand{\Crt}{\operatorname{crt}}
\newcommand{\Cr}{\operatorname{cr}}

\newcommand{\subtangpicture}[1][]{% 1 optional parameter for options for the tikz picture
\begin{center}
\begin{tikzpicture}[#1]
        \node[fill=black,circle,inner sep=2pt]  at (0,0) {};
        \draw[line width=2pt] (2,0)--(1,.5);
        \draw (0,0)--(2,-1);
         \draw (0,0)--(2,1);       
       \draw (6,0)--(4,-1);
        \draw (4,1)--(6,0);        
        \draw[line width=2pt] (5,-.5)--(4,0);
        \node[fill=black,circle,inner sep=2pt]  at (6,0) {};
        \draw[dashed] (2,1)--(4,1);
        \draw[dashed] (2,0)--(4,0);
        \draw[dashed] (2,-1)--(4,-1);
        \node[fill=black,circle,inner sep=1pt]  at (1,.5) {};
        \node[fill=black,circle,inner sep=1pt]  at (5,-.5) {};        
        \node[fill=black,rectangle,inner sep=2pt] at (2,1) {};
        \node[fill=black,rectangle,inner sep=2pt] at (2,0) {};
        \node[fill=black,rectangle,inner sep=2pt] at (2,-1) {};       
        \node[fill=black,rectangle,inner sep=2pt] at (4,1) {};
        \node[fill=black,rectangle,inner sep=2pt] at (4,0) {};
        \node[fill=black,rectangle,inner sep=2pt] at (4,-1) {};                 
	\node at (.5,1) {$S$};
\end{tikzpicture}
\end{center}
}

\newcommand{\old}[1]{{}}
\title{A tanglegram Kuratowski theorem}

\author{\'Eva Czabarka, L\'aszl\'o A. Sz\'ekely }
\address{\'Eva Czabarka and L\'aszl\'o A. Sz\'ekely\\ Department of Mathematics \\ University of South Carolina \\ Columbia, SC 29208 \\ USA}
\email{\{czabarka,szekely\}@math.sc.edu }
\thanks{The second author was supported in part by the  NSF DMS,  grant numbers 1300547 and 1600811,  the third author was supported by the National
Research Foundation of South Africa, grant number 96236.}
\author{Stephan Wagner}
\address{Stephan Wagner\\ Department of Mathematical Sciences \\ Stellenbosch University \\ Private Bag X1, Matieland 7602 \\ South Africa}
\email{swagner@sun.ac.za}

\begin{document}

\begin{abstract}
A tanglegram consists of two rooted binary  plane trees with the same number of leaves and a perfect matching between the two leaf sets. Tanglegrams are drawn with the leaves on two parallel lines, the trees on either side of the strip created by these lines, and the perfect matching inside the strip. If this can be done without any edges crossing, a tanglegram is called planar. We show that every non-planar tanglegram contains one of
two non-planar 4-leaf tanglegrams as induced subtanglegram, which parallels Kuratowski's Theorem.
\end{abstract}

\maketitle

\section{introduction}
Kuratowski's Theorem \cite{Kurat}, a cornerstone of graph theory, asserts that a graph is non-planar if and only if it contains a subdivision
of $K_{3,3} $ or $K_5$. This is not the only characterization of planarity: Wagner's Theorem \cite{wagner} asserts that a  
graph is non-planar if and only if 
$K_{3,3} $ or $K_5$ is a minor of the graph. Tanglegrams are a special kind of graph, consisting of two binary trees of the same size and a perfect matching joining the leaves, with restrictions on how they can be drawn.
Tanglegrams are well studied objects in phylogenetics and computer science. Planarity of tanglegrams is directly characterized by
Kuratowski's Theorem in terms of subgraphs, if the tanglegram is augmented by a certain edge (Lemma~\ref{lem:crt}). In this paper we provide a characterization
of planarity of tanglegrams not in terms of  its subgraphs, but in terms of other tanglegrams (Theorem~\ref{thm:main}). As tanglegrams are not widely
known objects, we immediately  turn to the technical definitions.   

%Here and in the following, whenever we write ``binary tree'', we refer to rooted binary trees, i.e., every vertex 
%is either a leaf or it 
%has exactly two children or none. The number of leaves of a rooted binary tree $T$ is denoted by $|T|$. 

%We consider two binary trees identical, if there is a graph isomorphism between them that maps root to root.

\section{Tanglegrams}

%Informally, a tanglegram consists of a pair of binary trees with the same number of leaves and a perfect matching joining the leaves of the two trees.
%Let us also give some formal definitions. 
A {\em plane binary tree} has a root vertex
assumed to be a common ancestor of all other vertices, and each vertex either has two children
(left and right) or no children. A vertex with no children is a  {\em leaf}, and a vertex with two
children is an {\em internal vertex}. Note that this definition allows a single-vertex tree that is considered as both root and leaf to be a rooted binary tree.
%It is well known that the number of plane binary trees with $n$ leaves is the Catalan number $C_n=\frac{1}{n} \binom{2n-2}{n-1}$. 
A plane binary tree is easy to draw  on one side 
of a line, without edge crossings, such that only the leaves of the tree are on the line.

\begin{figure}[htbp]
\begin{center}
\begin{tikzpicture}[scale=.67]
        \node[fill=black,circle,inner sep=1.5pt]  at (-1,0) {}; %*
        \node[fill=black,circle,inner sep=1.5pt]  at (1,-1) {};
        \node[fill=black,circle,inner sep=1.5pt]  at (1,1) {};
        \node[fill=black,rectangle,inner sep=2pt]  at (2,-1.5) {};
        \node[fill=black,rectangle,inner sep=2pt]  at (2,-.5) {};
        \node[fill=black,rectangle,inner sep=2pt]  at (2,.5) {};
        \node[fill=black,rectangle,inner sep=2pt]  at (2,1.5) {};

	\draw (-1,0)--(2,1.5);
	\draw (-1,0)--(2,-1.5);
	\draw (1,-1)--(2,-.5);
	\draw (1,1)--(2,.5);

        \node[fill=black,rectangle,inner sep=2pt]  at (3,-1.5) {};
        \node[fill=black,rectangle,inner sep=2pt]  at (3,-.5) {};
        \node[fill=black,rectangle,inner sep=2pt]  at (3,.5) {};
        \node[fill=black,rectangle,inner sep=2pt]  at (3,1.5) {};
        \node[fill=black,circle,inner sep=1.5pt]  at (4,1) {};
        \node[fill=black,circle,inner sep=1.5pt]  at (5,.5) {};
        \node[fill=black,circle,inner sep=1.5pt]  at (6,0) {}; %*

	\draw (6,0)--(3,1.5);
	\draw (6,0)--(3,-1.5);
	\draw (5,.5)--(3,-.5);
	\draw (4,1)--(3,.5);

	\draw [dashed] (2,-1.5)--(3,-1.5);
	\draw [dashed] (2,-.5)--(3,.5);
	\draw [dashed] (3,-.5)--(2,.5);
	\draw [dashed] (2,1.5)--(3,1.5);
	
	\node at (2,-1.9) {$a$};
	\node at (3,-1.9) {$a$};
         \node at (2,-.9)  {$b$};
         \node at (3,.1)  {$b$};
         \node at (3,-.9)  {$c$};
         \node at (2,.1)  {$c$};
         \node at  (2,1.1) {$d$};
         \node at (3,1.1)    {$d$};
	\node at (2.5,-2.5) {original layout};
	\node at (-1,0.3) {$r$};
	\node at (6,0.3) {$\rho$};
         
        \node[fill=black,circle,inner sep=1.5pt]  at (7,0) {}; 
        \node[fill=black,circle,inner sep=1.5pt]  at (9,-1) {};
        \node[fill=black,circle,inner sep=1.5pt]  at (9,1) {};
        \node[fill=black,rectangle,inner sep=2pt]  at (10,-1.5) {};
        \node[fill=black,rectangle,inner sep=2pt]  at (10,-.5) {};
        \node[fill=black,rectangle,inner sep=2pt]  at (10,.5) {};
        \node[fill=black,rectangle,inner sep=2pt]  at (10,1.5) {};

	\draw (7,0)--(10,1.5);
	\draw (7,0)--(10,-1.5);
	\draw (9,-1)--(10,-.5);
	\draw (9,1)--(10,.5);

        \node[fill=black,rectangle,inner sep=2pt]  at (11,-1.5) {};
        \node[fill=black,rectangle,inner sep=2pt]  at (11,-.5) {};
        \node[fill=black,rectangle,inner sep=2pt]  at (11,.5) {};
        \node[fill=black,rectangle,inner sep=2pt]  at (11,1.5) {};
        \node[fill=black,circle,inner sep=1.5pt]  at (12,0) {};
        \node[fill=black,circle,inner sep=1.5pt]  at (13,-.5) {};
        \node[fill=black,circle,inner sep=1.5pt]  at (14,0) {}; 

	\draw (14,0)--(11,1.5);
	\draw (14,0)--(11,-1.5);
	\draw (13,-.5)--(11,.5);
	\draw (12,0)--(11,-.5);

	\draw [dashed] (10,-1.5)--(11,1.5);
	\draw [dashed] (10,-.5)--(11,-.5);
	\draw [dashed] (10,.5)--(11,-1.5);
	\draw [dashed] (10,1.5)--(11,.5);
	
	\node at (10,-1.9) {$a$};
	\node at (10,-.9) {$b$};
	\node at (10,.1) {$c$};
	\node at (10,1.1) {$d$};
	\node at (11,1.1) {$a$};
	\node at (11,0.1) {$d$};
	\node at (11,-.9) {$b$};
	\node at (11,-1.9) {$c$};
	\node at (10.5,-2.5) {after switching at $\rho$};
	\node at (7,0.3) {$r$};
	\node at (14,0.3) {$\rho$};

        \node[fill=black,circle,inner sep=1.5pt]  at (15,0) {};
        \node[fill=black,circle,inner sep=1.5pt]  at (17,-1) {};
        \node[fill=black,circle,inner sep=1.5pt]  at (17,1) {};
        \node[fill=black,rectangle,inner sep=2pt]  at (18,-1.5) {};
        \node[fill=black,rectangle,inner sep=2pt]  at (18,-.5) {};
        \node[fill=black,rectangle,inner sep=2pt]  at (18,.5) {};
        \node[fill=black,rectangle,inner sep=2pt]  at (18,1.5) {};

	\draw (15,0)--(18,1.5);
	\draw (15,0)--(18,-1.5);
	\draw (17,-1)--(18,-.5);
	\draw (17,1)--(18,.5);

        \node[fill=black,rectangle,inner sep=2pt]  at (19,-1.5) {};
        \node[fill=black,rectangle,inner sep=2pt]  at (19,-.5) {};
        \node[fill=black,rectangle,inner sep=2pt]  at (19,.5) {};
        \node[fill=black,rectangle,inner sep=2pt]  at (19,1.5) {};
        \node[fill=black,circle,inner sep=1.5pt]  at (20,1) {};
        \node[fill=black,circle,inner sep=1.5pt]  at (21,.5) {};
        \node[fill=black,circle,inner sep=1.5pt]  at (22,0) {};

	\draw (22,0)--(19,1.5);
	\draw (22,0)--(19,-1.5);
	\draw (21,.5)--(19,-.5);
	\draw (20,1)--(19,.5);

	\draw [dashed] (18,-1.5)--(19,1.5); 
	\draw [dashed] (18,-.5)--(19,-.5); 
	\draw [dashed] (19,.5)--(18,.5);
	\draw [dashed] (18,1.5)--(19,-1.5);
	
	\node at (18,-1.9) {$d$};
	\node at (19,-1.9) {$a$};
         \node at (19,-.9)  {$c$};
         \node at (19,.1)  {$b$};
         \node at (18,-.9)  {$c$};
         \node at (18,.1)  {$b$};
         \node at  (18,1.1) {$a$};
         \node at (19,1.1)    {$d$};
         \node at (18.5,-2.5) {after mirroring at $r$};
         \node at (15,0.3) {$r$};
	\node at (22,0.3) {$\rho$};

\end{tikzpicture}
\end{center}
\caption{Results of a switch and a mirror  operation. 
%done at the root of the left binary tree. All layouts  in Figure~\ref{fig:switch} and 
%~\ref{fig:mirror} represent  No. 10 of Figure~\ref{fig:t4}.
}  \label{fig:sw&m}
\end{figure}
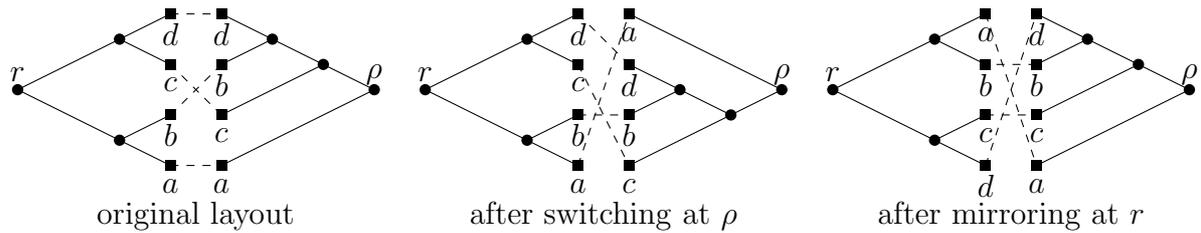

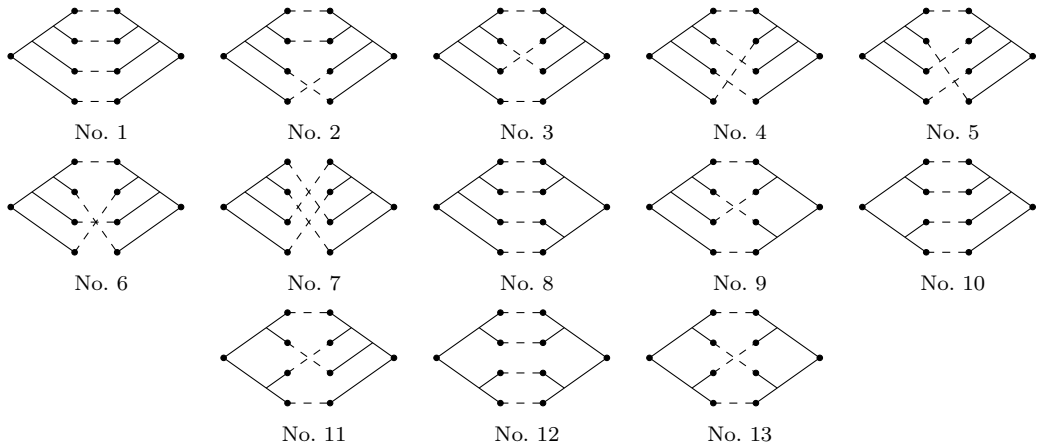
\begin{figure}[h!]
\begin{center}
\begin{tikzpicture}[scale = 0.4]
\newcommand{\treeb}[3]{\coordinate (v1) at (#1,#3); \coordinate (v2) at (#1+#2,#3+0.5);\coordinate (v3) at (#1+2*#2,#3+1);\coordinate (v4) at (#1+3*#2,#3+1.5);\coordinate (v5) at (#1+3*#2,#3+0.5);\coordinate (v6) at (#1+3*#2,#3-0.5);\coordinate (v7) at (#1+3*#2,#3-1.5);\draw[fill] (v1) circle (.5ex);\draw[fill] (v4) circle (.5ex);\draw[fill] (v5) circle (.5ex);\draw[fill] (v6) circle (.5ex);\draw[fill] (v7) circle (.5ex);\draw (v1) -- (v4);\draw (v3) -- (v5);\draw (v2) -- (v6);\draw (v1) -- (v7);}
\newcommand{\treec}[3]{\coordinate (v1) at (#1,#3); \coordinate (v2) at (#1+2*#2,#3-1);\coordinate (v3) at (#1+2*#2,#3+1);\coordinate (v4) at (#1+3*#2,#3+1.5);\coordinate (v5) at (#1+3*#2,#3+0.5);\coordinate (v6) at (#1+3*#2,#3-0.5);\coordinate (v7) at (#1+3*#2,#3-1.5);\draw[fill] (v1) circle (.5ex);\draw[fill] (v4) circle (.5ex);\draw[fill] (v5) circle (.5ex);\draw[fill] (v6) circle (.5ex);\draw[fill] (v7) circle (.5ex);\draw (v1) -- (v4);\draw (v3) -- (v5);\draw (v2) -- (v6);\draw (v1) -- (v7);}
\newcommand{\tangleb}[7]{\treeb{#1}{#2}{#7} \treeb{#1+8*#2}{-#2}{#7} \draw[dashed] (#1 + 3*#2,#7+1.5) -- (#1 + 5*#2,#7+2.5-#3); \draw[dashed] (#1 + 3*#2,#7+0.5) -- (#1 + 5*#2,#7+2.5-#4); \draw[dashed] (#1 + 3*#2,#7-0.5) -- (#1 + 5*#2,#7+2.5-#5); \draw[dashed] (#1 + 3*#2,#7-1.5) -- (#1 + 5*#2,#7+2.5-#6);}
\newcommand{\tanglec}[7]{\treeb{#1}{#2}{#7} \treec{#1+8*#2}{-#2}{#7} \draw[dashed] (#1 + 3*#2,#7+1.5) -- (#1 + 5*#2,#7+2.5-#3); \draw[dashed] (#1 + 3*#2,#7+0.5) -- (#1 + 5*#2,#7+2.5-#4); \draw[dashed] (#1 + 3*#2,#7-0.5) -- (#1 + 5*#2,#7+2.5-#5); \draw[dashed] (#1 + 3*#2,#7-1.5) -- (#1 + 5*#2,#7+2.5-#6);}
\newcommand{\tangled}[7]{\treec{#1}{#2}{#7} \treeb{#1+8*#2}{-#2}{#7} \draw[dashed] (#1 + 3*#2,#7+1.5) -- (#1 + 5*#2,#7+2.5-#3); \draw[dashed] (#1 + 3*#2,#7+0.5) -- (#1 + 5*#2,#7+2.5-#4); \draw[dashed] (#1 + 3*#2,#7-0.5) -- (#1 + 5*#2,#7+2.5-#5); \draw[dashed] (#1 + 3*#2,#7-1.5) -- (#1 + 5*#2,#7+2.5-#6);}
\newcommand{\tanglee}[7]{\treec{#1}{#2}{#7} \treec{#1+8*#2}{-#2}{#7} \draw[dashed] (#1 + 3*#2,#7+1.5) -- (#1 + 5*#2,#7+2.5-#3); \draw[dashed] (#1 + 3*#2,#7+0.5) -- (#1 + 5*#2,#7+2.5-#4); \draw[dashed] (#1 + 3*#2,#7-0.5) -- (#1 + 5*#2,#7+2.5-#5); \draw[dashed] (#1 + 3*#2,#7-1.5) -- (#1 + 5*#2,#7+2.5-#6);}
\tangleb {0} {0.7} 1 2 3 4 {0}
\tangleb {7} {0.7} 1 2 4 3 {0}
\tangleb {14} {0.7} 1 3 2 4 {0}
\tangleb {21} {0.7} 1 3 4 2 {0}
\tangleb {28} {0.7} 1 4 2 3 {0}
\tangleb {0} {0.7} 1 4 3 2 {-5}
\tangleb {7} {0.7} 3 4 1 2 {-5}
\tanglec {14} {0.7} 1 2 3 4 {-5}
\tanglec {21} {0.7} 1 3 2 4 {-5}
\tangled {28} {0.7} 1 2 3 4 {-5}
\tangled {7} {0.7} 1 3 2 4 {-10}
\tanglee {14} {0.7} 1 2 3 4 {-10}
\tanglee {21} {0.7} 1 3 2 4 {-10}
\node at (3,-2.5){$_{\text{No.\ } 1}$};
\node at (10,-2.5){$_{\text{No.\ } 2}$};
\node at (17,-2.5){$_{\text{No.\ } 3}$};
\node at (24,-2.5){$_{\text{No.\ } 4}$};
\node at (31,-2.5){$_{\text{No.\ } 5}$};
\node at (3,-7.5){$_{\text{No.\ } 6}$};
\node at (10,-7.5){$_{\text{No.\ } 7}$};
\node at (17,-7.5){$_{\text{No.\ } 8}$};
\node at (24,-7.5){$_{\text{No.\ } 9}$};
\node at (31,-7.5){$_{\text{No.\ } 10}$};
\node at (10,-12.5){$_{\text{No.\ } 11}$};
\node at (17,-12.5){$_{\text{No.\ } 12}$};
\node at (24,-12.5){$_{\text{No.\ } 13}$};
\end{tikzpicture}
\end{center}
\caption{The $13$ tanglegrams of size $4$ from \cite{KonvWag}.}\label{fig:t4}
\end{figure}

A  {\em tanglegram layout} $(L,R,\sigma)$ consists of a left plane binary tree $L$  with root $r$ drawn in the halfplane $x\leq 0$, having
its leaves on the $x=0$ line,
a right plane binary tree $R$ with root $\rho$, drawn in the halfplane $x\geq 1$, having
its leaves on the $x=1$ line,    each with $n$ leaves, 
and a perfect matching  $\sigma$ between their leaves drawn in straight line segments.
%(see Figure~\ref{fig:two_drawings}). 
 We treat tanglegram layouts
combinatorially, and understand them as  ordered triplets $(L,R,\sigma)$.
A {\em switch} on the tanglegram layout $(L,R,\sigma)$ is the following operation: select an internal vertex $v$ of one of the two trees $L$ and $R$. Vertex $v$  in $L$ has an up-subtree $L_u$ and a down-subtree $L_d$, with leaf sets $L(L_u)$ and $L(L_d)$ (or 
vertex $v$  in $R$ has an up-subtree $R_u$ and a down-subtree $R_d$, with leaf sets $L(R_u)$ and $L(R_d)$).
In the first case  interchange $L_u$ and 
$L_d$ such that on the line $x=0$  the order of two leaves changes precisely when one is in $L(L_u)$  and the other is in $L(L_d)$. In the 
second  case  interchange $R_u$ and 
$R_d$ such that on the line $x=1$  the order of two leaves changes precisely when one is in $L(R_u)$  and the other is in $L(R_d)$.
The edges of the matching move with the leaves that they connect during the switch. This is also illustrated in Figure~\ref{fig:sw&m}. 

Two layouts represent the same tanglegram if a sequence of switches  moves one layout into the other. For an internal vertex $v$, one may also take 
the {\em mirror image} of the subtree rooted at $v$. This is called the {\em mirror operation} at vertex $v$, which is illustrated in Figure~\ref{fig:sw&m}.
Matching edges still connect the corresponding leaves. As the mirror operation can be obtained by doing a sequence of switch operations from $v$ down
towards the leaves, mirror operations do not change the tanglegram. Tanglegrams partition the set of tanglegram layouts, or equivalently a tanglegram can be seen as an equivalence class of tanglegram layouts.
Note that interchanging $L$ and $R$ is not allowed and may result in a different tanglegram.  
%The tanglegram (i.e. the equivalence class)  of a layout  $(B_1,B_2,\sigma)$ will be denoted by $[(B_1,B_2,\sigma)]$. 

The {\em size} of a tanglegram is the number of leaves in $L$ (or $R$) in any of its layouts.
Billey, Konvalinka, and Matsen \cite{billey} considered the enumeration problem for tanglegrams: they obtained an explicit formula for the number $t_n$ of tanglegrams with $n$ leaves on each side. The counting sequence starts
$$1, 1, 2, 13, 114, 1509, 25595, 535753, 13305590, 382728552, \ldots.$$
Figure~\ref{fig:t4} illustrates the fourth term in this sequence. The asymptotic formula
$$t_n \sim n! \cdot \frac{e^{1/8} 4^{n-1}}{\pi n^3}$$
was derived in \cite{billey} as well, and a number of questions on the shape of random tanglegrams were asked. Those were answered in \cite{KonvWag} by means of a strong structure theorem.

\section{Tanglegram crossing number and planarity of tanglegrams}

The crossing number of a tanglegram layout is the number of crossing pairs of matching edges. 
The crossing number of a tanglegram layout does not depend on details of the drawing, such as the exact positions
of leaves on the vertical lines, just on the rankings of the  matched leaves in the  linear orders of leaves on 
the lines $x=0$ and $x=1$. This fact justifies the combinatorial treatment
of tanglegram layouts for studying crossings.

 It is desirable to draw a tanglegram with the {\em least possible number of crossings}, which is known as the 
 Tanglegram Layout Problem \cite{St.John}. This problem is NP-hard \cite{fernau}. 
The {\em (tanglegram) crossing number}  $\Crt(T)$ of a tanglegram $T$  is defined as the minimum number of crossings among its layouts. 

Tanglegrams play a major role in phylogenetics, especially in the theory of cospeciation. The first binary tree is the phylogenetic tree of hosts, while the second binary tree is the phylogenetic tree of their parasites, e.g. gopher and louse 
\cite{HafnerNadler}. The matching connects the host with its parasite.
 The tanglegram  crossing number  has been related to the  number of times  parasites switched hosts \cite{HafnerNadler},
or, working with gene trees instead of phylogenetic trees, to the 
number of horizontal gene transfers (\cite{Burt}, pp. 204--206).

A tanglegram is {\em planar} if it has zero tanglegram crossing number; in other words, if it has a layout without crossing matching edges. Otherwise it is called {\em non-planar}. In an earlier paper  \cite{induci} we showed  that the tanglegram crossing number
of a randomly and uniformly selected tanglegram of size $n$ is $\Theta(n^2)$ with high probability, i.e. as large as it can be within a constant
multiplicative factor. As one would therefore expect, the number of planar tanglegrams of size $n$ grows much more slowly than the total number of tanglegrams. The counting sequence $p_n$ starts
$$1, 1, 2, 11, 76, 649, 6173, 63429, 688898, 7808246, \ldots,$$
and an asymptotic formula of the form
$$p_n \sim A \cdot n^{-3} \cdot B^n,$$
where $A$ and $B$ are constants, holds \cite{RRW}.

Recall \cite{measure} that 
a {\em  drawing} of a graph $G$ in the plane places the
vertices of $G$ on distinct points in the plane and then, for every edge $uv$ in $G$, draws
a continuous simple curve in the plane connecting the two points corresponding to
$u$ and $v$, in such a way that no curve has a vertex point as an internal point.
The {\em crossing number} $\Cr(G)$ of a graph $G$ is the minimum number of intersection
points among the interiors of the curves representing edges, over all possible
drawings of the graph, where no three edges may have a common interior point.

Note that the (graph) crossing number $\Cr(T)$ of a tanglegram $T$ is less or equal to  the (tanglegram) crossing number  $\Crt(T)$  of $T$, since the tanglegram layout is more restrictive than the
graph drawing. The following lemma is essentially taken from \cite{induci}:
\begin{lemma}\label{lem:crt} 
Assume that a tanglegram $T$ is represented by the layout $(L,R,\sigma)$, and let the roots of $L$ and $R$ be $r$ and $\rho$. Let $T^*$ 
denote the graph in which the underlying graph of $T$ (consisting of the two binary trees and the matching edges)
is augmented by the edge $r\rho$. Then the following facts are equivalent:
\begin{enumerate}
\item $\Crt(T)\geq 1$,
\item $\Cr(T^*)\geq 1$,
\item $T^*$ contains a subdivision of $K_{3,3}$.
\end{enumerate}
\end{lemma}

\begin{proof}  (1)$\Rightarrow$(2). This is equivalent with $\Cr(T^*)=0\Rightarrow\Crt(T)=0$. If $\Cr(T^*)=0$ then $T^*$ can be drawn in the plane with
straight lines. This means both the left and right trees are drawn with straight lines, and we can draw a curve in the plane that goes through each matching edge once and no other edges of $T^*$ (the latter can be easily proven e.g. by induction on the number of leaves). Using the order of the leaves on this curve, one can easily obtain the desired planar layout of $T$.\\
(2)$\Rightarrow$(3) follows from Kuratowski's Theorem, as $T^*$ cannot contain a subdivision of $K_5$. This is because none of its vertices has degree greater than $3$. \\
(3)$\Rightarrow$(1) is proved by the contrapositive: if $T$ was to admit a planar tanglegram layout, then we could add the additional edge between $r$ and $\rho$, creating a planar drawing of the graph $T^*$.
\end{proof}

Note that a subdivision of a $K_{3,3}$ in the tanglegram $T^*$ may be such that the six vertices of $K_{3,3}$ are all in $L$ or all in $R$.

Now consider the tanglegrams $T_1$ and $T_2$ below,  No.~$6$ and $13$ in Figure~\ref{fig:t4}, each augmented with the extra edge connecting the roots. 
%Both are subdivisions of $K_4$ and therefore are planar graphs.
Figure \ref{fig:K3,3} shows a layout with one crossing for the tanglegrams  $T_1$ and $T_2$, and a subdivision of $K_{3,3}$ in the graphs
$T_1^*$ and $T_2^*$, showing $\Crt(T_1)=\Crt(T_2)=1$ in view of  Lemma~\ref{lem:crt}.

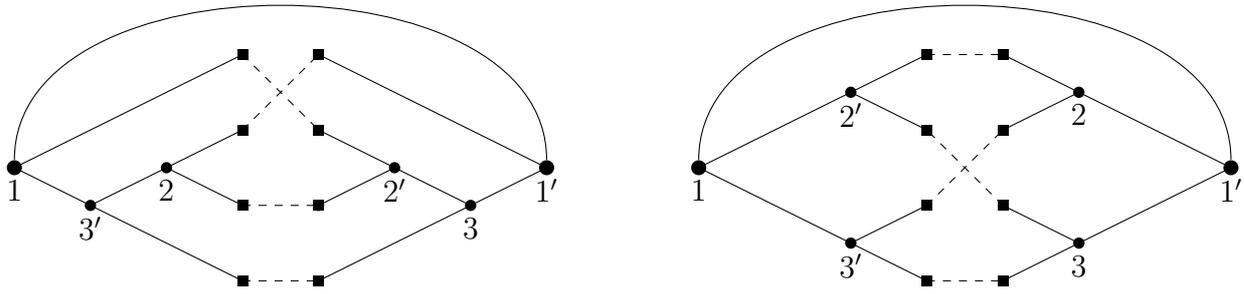
\begin{figure}[htbp]
\begin{center}
\begin{tikzpicture}[line/.style={-}]
  \node(A)[fill=black,circle,inner sep=2pt] at (8,0) {};
        \node[fill=black,circle,inner sep=1.5pt]  at (10,-1) {};
        \node[fill=black,circle,inner sep=1.5pt]  at (10,1) {};
        \node[fill=black,rectangle,inner sep=2pt]  at (11,-1.5) {};
        \node[fill=black,rectangle,inner sep=2pt]  at (11,-.5) {};
        \node[fill=black,rectangle,inner sep=2pt]  at (11,.5) {};
        \node[fill=black,rectangle,inner sep=2pt]  at (11,1.5) {};

	\draw (8,0)--(11,1.5);
	\draw (8,0)--(11,-1.5);
	\draw (10,-1)--(11,-.5);
	\draw (10,1)--(11,.5);

        \node[fill=black,rectangle,inner sep=2pt]  at (12,-1.5) {};
        \node[fill=black,rectangle,inner sep=2pt]  at (12,-.5) {};
        \node[fill=black,rectangle,inner sep=2pt]  at (12,.5) {};
        \node[fill=black,rectangle,inner sep=2pt]  at (12,1.5) {};
        \node[fill=black,circle,inner sep=1.5pt]  at (13,1) {};
        \node[fill=black,circle,inner sep=1.5pt]  at (13,-1) {};
        \node(B)[fill=black,circle,inner sep=2pt]  at (15,0) {};

	\draw (15,0)--(12,1.5);
	\draw (15,0)--(12,-1.5);
	\draw (13,-1)--(12,-.5);
	\draw (13,1)--(12,.5);
        \path[black,line,out=90,in=90] (A) edge (B);
        
        	\draw [dashed] (11,-1.5)--(12,-1.5);
	\draw [dashed] (11,-.5)--(12,.5);
	\draw [dashed] (12,-.5)--(11,.5);
	\draw [dashed] (11,1.5)--(12,1.5);

        \node(C)[fill=black,circle,inner sep=2pt]  at (-1,0) {};
       \node[fill=black,circle,inner sep=1.5pt]  at (0,-.5) {};
        \node[fill=black,circle,inner sep=1.5pt]  at (1,0) {};
        \node[fill=black,rectangle,inner sep=2pt]  at (2,-1.5) {};
        \node[fill=black,rectangle,inner sep=2pt]  at (2,-.5) {};
        \node[fill=black,rectangle,inner sep=2pt]  at (2,.5) {};
        \node[fill=black,rectangle,inner sep=2pt]  at (2,1.5) {};

	\draw (-1,0)--(2,1.5);
	\draw (-1,0)--(2,-1.5);
	\draw (0,-.5)--(1,0); %
	\draw (1,0)--(2,.5);
	\draw (1,0)--(2,-.5);

        \node[fill=black,rectangle,inner sep=2pt]  at (3,-1.5) {};
        \node[fill=black,rectangle,inner sep=2pt]  at (3,-.5) {};
        \node[fill=black,rectangle,inner sep=2pt]  at (3,.5) {};
        \node[fill=black,rectangle,inner sep=2pt]  at (3,1.5) {};
        \node[fill=black,circle,inner sep=1.5pt]  at (4,0) {};
        \node[fill=black,circle,inner sep=1.5pt]  at (3,-.5) {}; %
        \node[fill=black,circle,inner sep=1.5pt]  at (5,-.5) {};  %%
        \node(D)[fill=black,circle,inner sep=2pt]  at (6,0) {};

        \path[black,line,out=90,in=90] (C) edge (D);

	\draw (6,0)--(3,1.5);
	\draw (6,0)--(3,-1.5);
	\draw (5,-.5)--(3,.5);
	\draw (4,0)--(3,-.5);

	\draw [dashed] (2,-1.5)--(3,-1.5);
	\draw [dashed] (2,-.5)--(3,-.5);
	\draw [dashed] (2,.5)--(3,1.5);
	\draw [dashed] (2,1.5)--(3,.5);
        
        \node at (8,-.3) {$1$}; %
        \node at (10,.7) {$2'$};
        \node at (10,-1.3) {$3'$};
        \node at (13,.7) {$2$};
        \node at (13,-1.3) {$3$};
        \node at (15,-.3) {$1'$};

        \node at (6,-.3) {$1'$};
        \node at (5,-.8) {$3$};
        \node at  (4,-.3) {$2'$};
        \node at  (-1,-.3) {$1$};
        \node at (0,-.8) {$3'$};
        \node at (1,-.3) {$2$}; %

\end{tikzpicture}
\end{center}
\caption{Finding copies of $K_{3,3}$ in tanglegrams No.~6 and No.~13 after adding edges between the roots \cite{induci}.}\label{fig:K3,3}
\end{figure}

As it turns out, these two are the only non-planar tanglegrams of size $4$ (cf. Corollary~\ref{cor:t4remaining}). All others that have crossings in Figure~\ref{fig:t4}, can in fact be drawn without crossings. For example, Figure~\ref{fig:re-drawn} shows a crossing-free drawing of tanglegram No.~2.

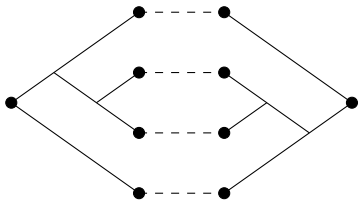
\begin{figure}[h!]
\begin{center}
\begin{tikzpicture}[scale = 0.8]

\coordinate (v0) at (0,0);
\coordinate (v1) at (0.7,0.5);
\coordinate (v2) at (1.4,0);
\coordinate (v3) at (2.1,1.5);
\coordinate (v4) at (2.1,0.5);
\coordinate (v5) at (2.1,-0.5);
\coordinate (v6) at (2.1,-1.5);

\coordinate (w0) at (5.6,0);
\coordinate (w1) at (4.9,-0.5);
\coordinate (w2) at (4.2,0);
\coordinate (w3) at (3.5,1.5);
\coordinate (w4) at (3.5,0.5);
\coordinate (w5) at (3.5,-0.5);
\coordinate (w6) at (3.5,-1.5);

\draw[fill] (v0) circle (.5ex);
\draw[fill] (v3) circle (.5ex);
\draw[fill] (v4) circle (.5ex);
\draw[fill] (v5) circle (.5ex);
\draw[fill] (v6) circle (.5ex);
\draw[fill] (w0) circle (.5ex);
\draw[fill] (w3) circle (.5ex);
\draw[fill] (w4) circle (.5ex);
\draw[fill] (w5) circle (.5ex);
\draw[fill] (w6) circle (.5ex);

\draw[dashed] (v3)--(w3);
\draw[dashed] (v4)--(w4);
\draw[dashed] (v5)--(w5);
\draw[dashed] (v6)--(w6);

\draw (v0)--(v3);
\draw (v0)--(v6);
\draw (v1)--(v5);
\draw (v2)--(v4);

\draw (w0)--(w3);
\draw (w0)--(w6);
\draw (w1)--(w4);
\draw (w2)--(w5);

\end{tikzpicture}
\end{center}
\caption{A drawing of tanglegram No.2 without crossings.}\label{fig:re-drawn}
\end{figure}

In the following, we will show that the two non-planar tanglegrams of size $4$ are in fact sufficient to characterize non-planarity of tanglegrams in the same way that $K_5$ and $K_{3,3}$ characterize non-planarity of graphs: every non-planar tanglegram has to contain at least one of the two.

%\end{proof}
%\begin{remark} It is easy to see from F\'ary's Theorem that $\Crt(T)=0$ if and only if $\Cr(T^*)=0$.
%\end{remark}

\section{Induced subtanglegrams}
In a rooted plane binary tree $B$ with root $r$, a choice $\mathcal{L}$ of a set of leaves induces another rooted binary tree by taking the smallest subtree containing these leaves and 
designating as new root (which we will denote by $r_{\mathcal{L}}$) the vertex of the subtree closest to the old root, and
suppressing all vertices of degree $2$ other than the root, see Figure~\ref{fig:example}. The study of induced binary subtrees of (rooted or unrooted) \emph{leaf-labeled} binary trees is topical in the phylogenetic literature
\cite{SempleSteel}. It 
is immediate that for any $v\in\mathcal{L}$ the vertex $r_{\mathcal{L}}$ lies on the unique path between $r$ and $v$. For brevity,
we  use the notation $r_{xy}$ instead of $r_{\{x,y\}}$.

\begin{figure}[htbp]
\begin{center}
\begin{tikzpicture}[scale=0.7]
        \node[fill=black,rectangle,inner sep=3pt]  at (0,0) {};
        \node[fill=black,rectangle,inner sep=3pt]  at (1.5,0) {};
        \node[fill=black,rectangle,inner sep=3pt]  at (3,0) {};
        \node[fill=black,rectangle,inner sep=3pt]  at (4.5,0) {};
        \node[fill=black,rectangle,inner sep=3pt]  at (6,0) {};
        \node[fill=black,rectangle,inner sep=3pt]  at (7.5,0) {};
        \node[fill=black,rectangle,inner sep=3pt]  at (9,0) {};
        \node[fill=black,rectangle,inner sep=3pt]  at (10.5,0) {};

        \node[fill=black,circle,inner sep=2.5pt]  at (5.25,5.25) {};
        \node[fill=black,circle,inner sep=2pt]  at (1.5,1.5) {};
        \node[fill=black,circle,inner sep=2pt]  at (8.25,2.25) {}; 
        \node[fill=black,circle,inner sep=2pt]  at (0.75,0.75) {};
        \node[fill=black,circle,inner sep=2pt]  at (7.5,3) {};
        \node[fill=black,circle,inner sep=2pt]  at (6.75,0.75) {};
        \node[fill=black,circle,inner sep=2pt]  at (9.75,0.75) {};

	\draw (0,0)--(5.25,5.25)--(10.5,0);
	\draw (4.5,0)--(7.5,3);
	\draw (6,0)--(8.25,2.25);
	\draw (9,0)--(9.75,0.75);
	\draw (1.5,0)--(0.75,0.75);
	\draw (3,0)--(1.5,1.5);
	\draw (7.5,0)--(6.75,0.75);
	\draw (10.5,0)--(9.75,0.75);

	\node at (4.5,-0.5) {$\ell_1$};
	\node at (6,-0.5) {$\ell_2$};
	\node at (9,-0.5) {$\ell_3$};
	\node at (10.5,-0.5) {$\ell_4$};
	\node at (5.75,5.25) {$r$};
	\node at (8.75,3) {$r_{\ell_1\ell_2\ell_3\ell_4}$};
%	\node at (5.25,-1.5) {The original tree $T$ with four leaves selected.};

	\draw (11, 1.5)--(12.5,1.5)--(12.5,1)--(13.5,1.75)--(12.5,2.5)--(12.5,2)--(11,2)--(11,1.5);

        \node[fill=black,rectangle,inner sep=3pt]  at (14,0) {};
        \node[fill=black,rectangle,inner sep=3pt]  at (15.5,0) {};
        \node[fill=black,rectangle,inner sep=3pt]  at (17,0) {};
        \node[fill=black,rectangle,inner sep=3pt]  at (18.5,0) {};

        \node[fill=black,circle,inner sep=2pt]  at (17,1.5) {};
        \node[fill=black,circle,inner sep=2.5pt]  at (16.25,2.25) {};
        \node[fill=black,circle,inner sep=2pt]  at (17.75,0.75) {};

	\draw (14,0)--(16.25,2.25)--(18.5,0);
	\draw (17.,1.5)--(15.5,0);
	\draw (17.75,0.75)--(17,0);

	\node at (14,-0.5) {$\ell_1$};
	\node at (15.5,-0.5) {$\ell_2$};
	\node at (17,-0.5) {$\ell_3$};
	\node at (18.5,-0.5) {$\ell_4$};

%	\node at (16.5,-1.5) {The tree induced by the selected leaves.};
\end{tikzpicture}
\end{center}
\caption{A rooted binary tree with root $r$, four leaves $\ell_1,\ell_2,\ell_3,\ell_4$ selected, the vertex $r_{\ell_1\ell_2\ell_3\ell_4}$ and the tree induced by the selected leaves.}\label{fig:example}
\end{figure}
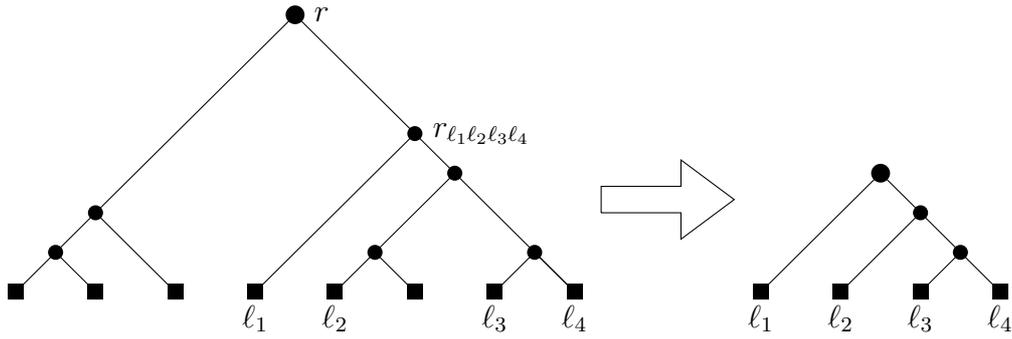

Given a  layout of the tanglegram $T$ with left binary tree $L$ with root $r$ and right binary tree $R$ with root $\rho$ and a set $E$ of matching edges
between the leaf sets of the left and right binary plane trees, $E$ identifies a subset of leaves on both sides. These leaf sets induce respectively
a left and right induced binary plane tree, which define a  layout of a tanglegram  $T'$ when we put back the edges of $E$ between the corresponding leaves.
We say that $E$ induces this   {\em  sublayout}  of the original layout of tanglegram $T$, and we call $T'$ the {\em subtanglegram of  tanglegram $T$} induced by the matching edge set $E$. As the sublayout and switch operators commute, this definition does not depend 
on the particular layout of $T$, it just depends on the tanglegram. We will use $r_E$ and $\rho_E$ for the vertices in $T$ corresponding to the roots of the left and right subtrees of this induced subtanglegram. There is a natural partial order by inclusion on the set of induced subtanglegrams of a given tanglegram.

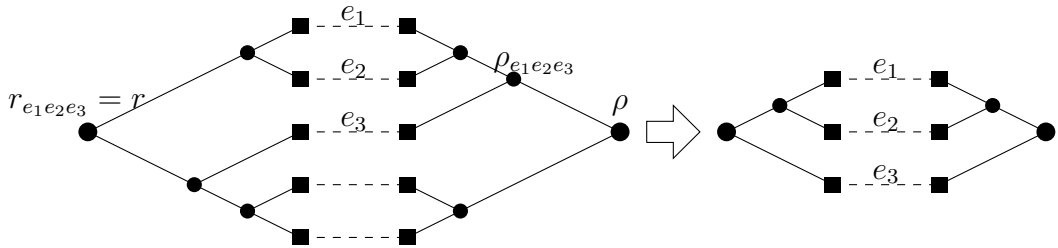
\begin{figure}[htbp]
\begin{center}
\begin{tikzpicture}[scale=.7]
        \node[fill=black,circle,inner sep=2.5pt]  at (0,0) {};
        \node[fill=black,circle,inner sep=2pt]  at (2,-1) {};
        \node[fill=black,circle,inner sep=2pt]  at (3,1.5) {};
        \node[fill=black,circle,inner sep=2pt]  at (3,-1.5) {};
        \node[fill=black,rectangle,inner sep=3pt]  at (4,2) {};
        \node[fill=black,rectangle,inner sep=3pt]  at (4,1) {};
        \node[fill=black,rectangle,inner sep=3pt]  at (4,0) {};
        \node[fill=black,rectangle,inner sep=3pt]  at (4,-1) {};
        \node[fill=black,rectangle,inner sep=3pt]  at (4,-2) {};
        \draw (4,-2)--(0,0)--(4,2);
        \draw (4,1)--(3,1.5);
        \draw (4,0)--(2,-1);
        \draw (4,-1)--(3,-1.5);
        \draw[dashed] (4,2)--(6,2);
        \draw[dashed] (4,1)--(6,1);
        \draw[dashed] (4,0)--(6,0);
        \draw[dashed] (4,-1)--(6,-1);
        \draw[dashed] (4,-2)--(6,-2);
       \node[fill=black,rectangle,inner sep=3pt]  at (6,2) {};
        \node[fill=black,rectangle,inner sep=3pt]  at (6,1) {};
        \node[fill=black,rectangle,inner sep=3pt]  at (6,0) {};
        \node[fill=black,rectangle,inner sep=3pt]  at (6,-1) {};
        \node[fill=black,rectangle,inner sep=3pt]  at (6,-2) {};
        \node[fill=black,circle,inner sep=2pt]  at (7,1.5) {};
        \node[fill=black,circle,inner sep=2pt]  at (7,-1.5) {};
        \node[fill=black,circle,inner sep=2pt]  at (8,1) {};
        \node[fill=black,circle,inner sep=2.5pt]  at (10,0) {};
        \draw (6,-2)--(10,0)--(6,2);
       \draw (6,1)--(7,1.5);
        \draw (6,0)--(8,1);
        \draw (6,-1)--(7,-1.5);
 	\node at (5,2.2) {$e_1$};
    	\node at (5,1.2) {$e_2$};    
  	\node at (5,.2) {$e_3$};
 	\node at (8.4,1.3) {$\rho_{e_1e_2e_3}$};	
 	\node at (-0.2,0.5) {$r_{e_1e_2e_3}=r$};
 	\node at (10,.5) {$\rho$};	
 	\draw (10.5, -0.2)--(11,-0.2)--(11,-0.5)--(11.5,0)--(11,.5)--(11,.2)--(10.5,.2)--(10.5,-.2);

        \node[fill=black,circle,inner sep=2.5pt]  at (12,0) {};
        \node[fill=black,circle,inner sep=2pt]  at (13,.5) {};
       \node[fill=black,rectangle,inner sep=3pt]  at (14,1) {};
       \node[fill=black,rectangle,inner sep=3pt]  at (14,0) {};
       \node[fill=black,rectangle,inner sep=3pt]  at (14,-1) {};
       \draw (14,1)--(12,0)--(14,-1);
       \draw (14,0)--(13,.5);
        \draw[dashed] (14,1)--(16,1);
        \draw[dashed] (14,0)--(16,0);
        \draw[dashed] (14,-1)--(16,-1);
        \draw (16,1)--(18,0)--(16,-1);
        \draw (16,0)--(17,.5);
        \node[fill=black,circle,inner sep=2.5pt]  at (18,0) {};
        \node[fill=black,circle,inner sep=2pt]  at (17,.5) {};	
        \node[fill=black,rectangle,inner sep=3pt]  at (16,1) {};
       \node[fill=black,rectangle,inner sep=3pt]  at (16,0) {};
       \node[fill=black,rectangle,inner sep=3pt]  at (16,-1) {};
  	\node at (15,1.2) {$e_1$};
    	\node at (15,.2) {$e_2$};    
  	\node at (15,.-.8) {$e_3$};
        
\end{tikzpicture}
\end{center}
\caption{A tanglegram $T$ with matching edges $e_1,e_2,e_3$ selected, the vertices $r_{e_1e_2e_3}$ and $\rho_{e_1,e_2,e_3}$, and the subtanglegram induced by the selected edges.}\label{fig:example2}
\end{figure}

Sometimes we put {\em scars} on the edges of induced subtanglegrams to remember where the eliminated matching edges were connected 
to the surviving part. Let $e\in \sigma\setminus E$ be a matching edge in a layout of the tanglegram $T$. When we consider $L_E$ and $R_E$, 
the smallest subtrees of
$L$ and $R$ that contain the leaf set corresponding to $E$, the unique path connecting $e$ to $r$ in $L$ either
enters $L_E$ at a vertex of degree 2 or does not enter $L_E$ at all, and similarly, the unique path connecting $e$ to $\rho$ in $R$ either enters $R_E$ at a vertex of degree 2 or does not enter $R_E$ at all. We refer to these degree $2$ vertices (when they exist) as the {\em hosts} of $e$ in $L_E$ and $R_E$. A single vertex can  host  several other matching edges not in $E$. Hosts in $L_E$ (respectively in $R_E$) are in a natural partial order by separation from $r$ (respectively $\rho$). 

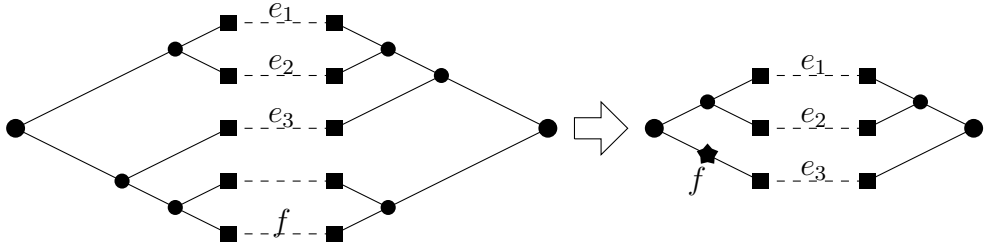
\begin{figure}[htbp]
\begin{center}
\begin{tikzpicture}[scale=.7]
        \node[fill=black,circle,inner sep=2.5pt]  at (0,0) {};
        \node[fill=black,circle,inner sep=2pt]  at (2,-1) {};
        \node[fill=black,circle,inner sep=2pt]  at (3,1.5) {};
        \node[fill=black,circle,inner sep=2pt]  at (3,-1.5) {};
        \node[fill=black,rectangle,inner sep=3pt]  at (4,2) {};
        \node[fill=black,rectangle,inner sep=3pt]  at (4,1) {};
        \node[fill=black,rectangle,inner sep=3pt]  at (4,0) {};
        \node[fill=black,rectangle,inner sep=3pt]  at (4,-1) {};
        \node[fill=black,rectangle,inner sep=3pt]  at (4,-2) {};
        \draw (4,-2)--(0,0)--(4,2);
        \draw (4,1)--(3,1.5);
        \draw (4,0)--(2,-1);
        \draw (4,-1)--(3,-1.5);
        \draw[dashed] (4,2)--(6,2);
        \draw[dashed] (4,1)--(6,1);
        \draw[dashed] (4,0)--(6,0);
        \draw[dashed] (4,-1)--(6,-1);
        \draw[dashed] (4,-2)--(6,-2);
       \node[fill=black,rectangle,inner sep=3pt]  at (6,2) {};
        \node[fill=black,rectangle,inner sep=3pt]  at (6,1) {};
        \node[fill=black,rectangle,inner sep=3pt]  at (6,0) {};
        \node[fill=black,rectangle,inner sep=3pt]  at (6,-1) {};
        \node[fill=black,rectangle,inner sep=3pt]  at (6,-2) {};
        \node[fill=black,circle,inner sep=2pt]  at (7,1.5) {};
        \node[fill=black,circle,inner sep=2pt]  at (7,-1.5) {};
        \node[fill=black,circle,inner sep=2pt]  at (8,1) {};
        \node[fill=black,circle,inner sep=2.5pt]  at (10,0) {};
        \draw (6,-2)--(10,0)--(6,2);
       \draw (6,1)--(7,1.5);
        \draw (6,0)--(8,1);
        \draw (6,-1)--(7,-1.5);
 	\node at (5,2.2) {$e_1$};
    	\node at (5,1.2) {$e_2$};    
  	\node at (5,.2) {$e_3$};
  	\node at (5,-1.8) {$f$};
 	\draw (10.5, -0.2)--(11,-0.2)--(11,-0.5)--(11.5,0)--(11,.5)--(11,.2)--(10.5,.2)--(10.5,-.2);

        \node[fill=black,circle,inner sep=2.5pt]  at (12,0) {};
        \node[fill=black,circle,inner sep=2pt]  at (13,.5) {};
       \node[fill=black,rectangle,inner sep=3pt]  at (14,1) {};
       \node[fill=black,rectangle,inner sep=3pt]  at (14,0) {};
       \node[fill=black,rectangle,inner sep=3pt]  at (14,-1) {};
       \draw (14,1)--(12,0)--(14,-1);
       \draw (14,0)--(13,.5);
        \draw[dashed] (14,1)--(16,1);
        \draw[dashed] (14,0)--(16,0);
        \draw[dashed] (14,-1)--(16,-1);
        \draw (16,1)--(18,0)--(16,-1);
        \draw (16,0)--(17,.5);
        \node[fill=black,circle,inner sep=2.5pt]  at (18,0) {};
        \node[fill=black,circle,inner sep=2pt]  at (17,.5) {};	
        \node[fill=black,rectangle,inner sep=3pt]  at (16,1) {};
       \node[fill=black,rectangle,inner sep=3pt]  at (16,0) {};
       \node[fill=black,rectangle,inner sep=3pt]  at (16,-1) {};
  	\node at (15,1.2) {$e_1$};
    	\node at (15,.2) {$e_2$};    
  	\node at (15,.-.8) {$e_3$};
       \node[fill=black,star,star points=5,inner sep=2pt]  at (13,-.5) {};
  	\node at (12.8,-1) {$f$};
        
\end{tikzpicture}
\end{center}
\caption{The scar of edge $f$ in the example of Figure~\ref{fig:example2}. Notice that the right tree does not have a scar for $f$.}\label{fig:scar}
\end{figure}

{\em Scars} are markings on the edges of the induced subtanglegram, corresponding to the host vertices and following the natural partial order above, such that every scar marks the names of all edges hosted. Note that the partial order of the scars and the corresponding marks do not depend on the layout, they only depend on the tanglegram. Figure~\ref{fig:scar} illustrates a scar.

The following lemmata will be used to prove our main result:
\begin{lemma} \label{lem:kukac}
If $F$ is a planar %layout of a 
tanglegram with a distinguished marked non-matching edge, such that in every planar layout of
$F$, the marked edge does not lie on the boundary  of the infinite face, then $F$ has a set of three edges $E$ that induces
the following subtanglegram $S$, where the marked edge lies on one of the paths of $F$ corresponding to the bold edges.
\begin{figure}[h]
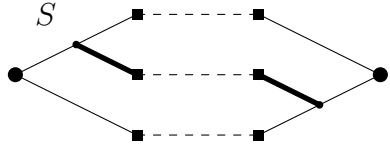

\begin{center}
\subtangpicture[scale=.8]
\end{center}
\caption{The subtanglegram $S$.}
\end{figure}
\end{lemma}

\begin{proof} Let the left and right tree of $F$ be $L$ and $R$ with roots $r$ and $\rho$ respectively, and $\sigma$ denote the set of matching edges.
We denote the marked edge by $m$ and assume without loss of generality that $m$ is an edge of $R$ (the argument is the same otherwise with the roles of $L$ and $R$ exchanged). Consider the unique path $P$ in $R$ that starts from $\rho$ and whose last edge is $m$. Let $m^*$ be the edge
of $P$ closest to $\rho$ that does not lie on the boundary of any planar layout of $F$ (potentially $m^*=m$). 
   
 \begin{figure}[htbp]
\begin{center}
\begin{tikzpicture}

	\node at (6.5,-.25) {$r$};
	\node at (7,-.6) {$r_{e_1e_2}$};
	\node at (9.7,-.65) {$\rho^*$};
	\node at (10.5,.3) {$\rho$};
        \node(A2)[fill=black,circle,inner sep=2.5pt]  at (6.5,0) {};
        \node(B2)[fill=black,circle,inner sep=2.5pt]  at (10.25,0) {};        
        \node(C2)[fill=black,circle,inner sep=2pt]  at (7.25,-.375) {};
        \node(D2)[fill=black,circle,inner sep=2pt]  at (9.5,-.375) {};        
        \node(E2)[fill=black,rectangle,inner sep=2.5pt] at (8,.75) {};
        \node(F2)[fill=black,rectangle,inner sep=2.5pt] at (8,0) {};
        \node(G2)[fill=black,rectangle,inner sep=2.5pt] at (8,-.75) {};       
        \node(H2)[fill=black,rectangle,inner sep=2.5pt] at (8.75,.75) {};
        \node(J2)[fill=black,rectangle,inner sep=2.5pt] at (8.75,0) {};
        \node(K2)[fill=black,rectangle,inner sep=2.5pt] at (8.75,-.75) {};                 
      
        \draw (E2) -- (A2);
        \draw (A2) -- (C2);
         \draw (C2) -- (G2);
         \draw (C2) -- (F2);       
        \draw (B2)--(D2); 
        \draw[line width =2] (D2) --(J2);
        \draw (H2) -- (B2);        
       \draw (D2) --(K2);
        \draw[dashed] (E2)--(H2);
        \draw[dashed] (F2)--(J2);
        \draw[dashed] (G2)--(K2);
	\node at (8.375,.95) {$e_3$};
	\node at (8.375,.25) {$e_2$};
	\node at  (8.375,-.55) {$e_1$};

%	\draw (5.125,-2)--(5.125,5); %separating line
	
	\node at (0,-.25) {$r$};
	\node at (.4,.7) {$r_{e_2e_3}$};
	\node at (3.2,-.65) {$\rho^*$};
	\node at (3.9,.3) {$\rho$};
        \node(A3)[fill=black,circle,inner sep=2.5pt]  at (0,0) {};
        \node(B3)[fill=black,circle,inner sep=2.5pt]  at (3.75,0) {};        
        \node(C3)[fill=black,circle,inner sep=2pt]  at (.75,.375) {};
        \node(D3)[fill=black,circle,inner sep=2pt]  at (3,-.375) {};        
        \node(E3)[fill=black,rectangle,inner sep=2.5pt] at (1.5,.75) {};
        \node(F3)[fill=black,rectangle,inner sep=2.5pt] at (1.5,0) {};
        \node(G3)[fill=black,rectangle,inner sep=2.5pt] at (1.5,-.75) {};       
        \node(H3)[fill=black,rectangle,inner sep=2.5pt] at (2.25,.75) {};
        \node(J3)[fill=black,rectangle,inner sep=2.5pt] at (2.25,0) {};
        \node(K3)[fill=black,rectangle,inner sep=2.5pt] at (2.25,-.75) {};                 
      
        \draw (E3) --(C3);
        \draw (A3) --(G3);
         \draw (A3) -- (C3);
         \draw (C3) -- (F3);       
        \draw (B3)--(D3); 
        \draw[line width=2pt] (D3)-- (J3);
        \draw (H3) -- (B3);        
       \draw (D3) -- (K3);
        \draw[dashed] (E3)--(H3);
        \draw[dashed] (F3)--(J3);
        \draw[dashed] (G3)--(K3);
	\node at (1.875,.95) {$e_3$};
	\node at (1.875,0.25) {$e_2$};
	\node at  (1.875,-0.55) {$e_1$};

\end{tikzpicture}
\end{center}
\caption{The possible subtanglegrams of $F$ induced by $e_1,e_2,e_3$. The edge containing $m^*$ is bold.}\label{fig:szembe}
\end{figure}
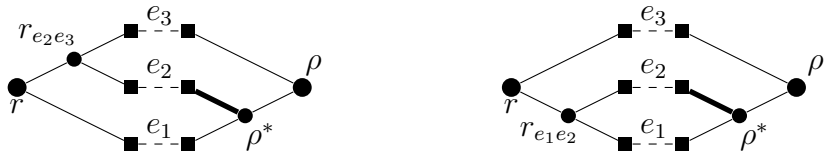

Consider a planar layout of $F$ where one endpoint of $m^*$ (which we will denote by $\rho^*$) lies on the boundary of the infinite face; by the definition of $m^*$ such a layout exists. Without loss of generality it is the lower of the two $r$-$\rho$ paths on the boundary. Let $E^*$ be the set of matching edges on the leaves of the subtree of $R$ rooted at $\rho^*$. By our assumptions $\rho\ne\rho^*$, $|E^*|\ge 2$, $E^*\ne \sigma$, and all edges of $E^*$ lie below all edges of $M\setminus E^*$
in the layout.
Let $e_1\in E^*$ and $e_3\in \sigma\setminus E^*$ be the matching edges of $F$ that are on the boundary of the infinite face of the layout  and let $e_2\in E^*$ be the edge that lies above all other edges of $E^*$ in the layout; so we have
 $\rho=\rho_{e_1e_3}$, $r=r_{e_1e_3}$ and $\rho^*=\rho_{e_1e_2}$ (See  Figure~\ref{fig:szembe}). We must have
 $r\in\{r_{e_1e_2},r_{e_2e_3}\}$. If $r=r_{e_2e_3}$, then $r_{e_1e_2}$ lies on the unique $r$-$e_1$ path in $L$, and performing a mirror
 operation on $r_{e_1e_2}$ and $\rho^*$ results in a planar layout of $F$ where $m^*$ lies on the boundary of the infinite face, which is a contradiction. Therefore we must have $r=r_{e_1e_2}$, and $r_{e_2e_3}$ lies on the unique $r$-$e_3$ path in $L$.

 \begin{figure}[htbp]
\begin{center}
\begin{tikzpicture}
	\node at (0,-.25) {$r$};
	\node at (3.4,-.7) {$\rho^*$};
	\node at (3.9,.3) {$\rho$};
        \node(A3)[fill=black,circle,inner sep=2.5pt]  at (0,0) {};
        \node(B3)[fill=black,circle,inner sep=2.5pt]  at (3.75,0) {};        
        \node(C3)[fill=black,circle,inner sep=2pt]  at (.75,.75) {}; 
        \node(D3)[fill=black,circle,inner sep=2pt]  at (3.25,-.45) {};         
        \node(E3)[fill=black,rectangle,inner sep=2.5pt] at (1.5,1.125) {};
        \node(F3)[fill=black,rectangle,inner sep=2.5pt] at (1.5,.375) {};
        \node(G3)[fill=black,rectangle,inner sep=2.5pt] at (1.5,-1.15) {};       
        \node(H3)[fill=black,rectangle,inner sep=2.5pt] at (2.25,1.125) {};
        \node(J3)[fill=black,rectangle,inner sep=2.5pt] at (2.25,.375) {};
        \node(K3)[fill=black,rectangle,inner sep=2.5pt] at (2.25,-1.125) {};                  
        \node(L3)[fill=black,rectangle,inner sep=2.5pt] at (1.5,-.375) {};       
        \node(M3)[fill=black,rectangle,inner sep=2.5pt] at (2.25,-.375) {};       
        \node(N3)[fill=black,circle,inner sep=2pt] at (2.75,0) {};       
        \node(O3)[fill=black,circle,inner sep=2pt] at (.375,.375) {};               
        \draw (E3) --(C3);
        \draw (A3) --(G3);
         \draw (A3) -- (O3)--(C3);
         \draw (C3) -- (F3);       
        \draw (B3)--(D3); 
        \draw (D3)--(N3)-- (J3);
        \draw (H3) -- (B3);        
       \draw (D3) -- (K3);
       \draw (O3)--(L3);              
       \draw[line width=2] (N3)--(M3);
        \draw[dashed] (E3)--(H3);
        \draw[dashed] (F3)--(J3);
        \draw[dashed] (G3)--(K3);
         \draw[dashed] (L3)--(M3);
	\node at (1.875,-.2) {$f$};        
	\node at (1.875,1.3) {$e_3$};
	\node at (1.875,0.55) {$e_2$};
	\node at  (1.875,-1) {$e_1$};
	\node at (2.9,.3) {$\rho_{fe_2}$};

	\node at (4.5,-.25) {$r$};
	\node at (7.9,-.7) {$\rho^*$};
	\node at (8.4,.3) {$\rho$};
        \node(A2)[fill=black,circle,inner sep=2.5pt]  at (4.5,0) {};
        \node(B2)[fill=black,circle,inner sep=2.5pt]  at (8.25,0) {};        
        \node(C2)[fill=black,circle,inner sep=2pt]  at (5.25,.75) {}; 
        \node(D2)[fill=black,circle,inner sep=2pt]  at (7.75,-.45) {};         
        \node(E2)[fill=black,rectangle,inner sep=2.5pt] at (6,1.125) {};
        \node(F2)[fill=black,rectangle,inner sep=2.5pt] at (6,.375) {};
        \node(G2)[fill=black,rectangle,inner sep=2.5pt] at (6,-1.15) {};       
        \node(H2)[fill=black,rectangle,inner sep=2.5pt] at (6.75,1.125) {};
        \node(J2)[fill=black,rectangle,inner sep=2.5pt] at (6.75,.375) {};
        \node(K2)[fill=black,rectangle,inner sep=2.5pt] at (6.75,-1.125) {};                  
        \node(L2)[fill=black,rectangle,inner sep=2.5pt] at (6,-.375) {};       
        \node(M2)[fill=black,rectangle,inner sep=2.5pt] at (6.75,-.375) {};       
        \node(N2)[fill=black,circle,inner sep=2pt] at (7.25,0) {};       
        \node(O2)[fill=black,circle,inner sep=2pt] at (5.625,.5615) {};       
        \draw (E2) --(C2);
        \draw (A2) --(G2);
         \draw (A2) -- (C2);
         \draw (C2) -- (O2)--(F2);       
        \draw (B2)--(D2); 
        \draw (D2)--(N2)-- (J2);
        \draw (H2) -- (B2);        
       \draw (D2) -- (K2);
       \draw (O2)--(L2);       
       \draw[line width=2] (N2)--(M2);
        \draw[dashed] (E2)--(H2);
        \draw[dashed] (F2)--(J2);
        \draw[dashed] (G2)--(K2);
         \draw[dashed] (L2)--(M2);
	\node at (6.375,-.2) {$f$};        
	\node at (6.375,1.3) {$e_3$};
	\node at (6.375,0.55) {$e_2$};
	\node at  (6.375,-1) {$e_1$};
	\node at (7.4,.3) {$\rho_{fe_2}$};

	\node at (4,-1.7) {when $r=r_{fe_1}$};

	\draw (8.65,-2)--(8.65,1.75);

	\node at (9,-.25) {$r$};
	\node at (12.4,-.7) {$\rho^*$};
	\node at (12.9,.3) {$\rho$};
        \node(A1)[fill=black,circle,inner sep=2.5pt]  at (9,0) {};
        \node(B1)[fill=black,circle,inner sep=2.5pt]  at (12.75,0) {};        
        \node(C1)[fill=black,circle,inner sep=2pt]  at (9.75,.75) {}; 
        \node(D1)[fill=black,circle,inner sep=2pt]  at (12.25,-.45) {};         
        \node(E1)[fill=black,rectangle,inner sep=2.5pt] at (10.5,1.125) {};
        \node(F1)[fill=black,rectangle,inner sep=2.5pt] at (10.5,.375) {};
        \node(G1)[fill=black,rectangle,inner sep=2.5pt] at (10.5,-1.15) {};       
        \node(H1)[fill=black,rectangle,inner sep=2.5pt] at (11.25,1.125) {};
        \node(J1)[fill=black,rectangle,inner sep=2.5pt] at (11.25,.375) {};
        \node(K1)[fill=black,rectangle,inner sep=2.5pt] at (11.25,-1.125) {};                  
        \node(L1)[fill=black,rectangle,inner sep=2.5pt] at (10.5,-.375) {};       
        \node(M1)[fill=black,rectangle,inner sep=2.5pt] at (11.25,-.375) {};       
        \node(N1)[fill=black,circle,inner sep=2pt] at (11.75,0) {};   
        \node(O1)[fill=black,circle,inner sep=2pt] at (10,-.75) {};   
        \draw (E1) --(C1);
        \draw (A1) --(O1)--(G1);
         \draw (A1) -- (C1);
         \draw (C1) -- (F1);       
        \draw (B1)--(D1); 
        \draw (D1)--(N1)-- (J1);
        \draw (H1) -- (B1);        
       \draw (D1) -- (K1);
       \draw (O1)--(L1);
       \draw[line width=2] (N1)--(M1);
        \draw[dashed] (E1)--(H1);
        \draw[dashed] (F1)--(J1);
        \draw[dashed] (G1)--(K1);
         \draw[dashed] (L1)--(M1);
	\node at (10.875,-.2) {$f$};        
	\node at (10.875,1.3) {$e_3$};
	\node at (10.875,0.55) {$e_2$};
	\node at  (10.875,-1) {$e_1$};
	\node at (11.9,.3) {$\rho_{fe_2}$};
	\node at (10.85, -1.7) {when $r=r_{fe_2}$};
\end{tikzpicture}
\end{center}
\caption{The possible subtanglegrams of $F$ induced by $e_1,e_2,e_3,f$. The edge containing $m$ is bold.}\label{fig:mcases}
\end{figure}
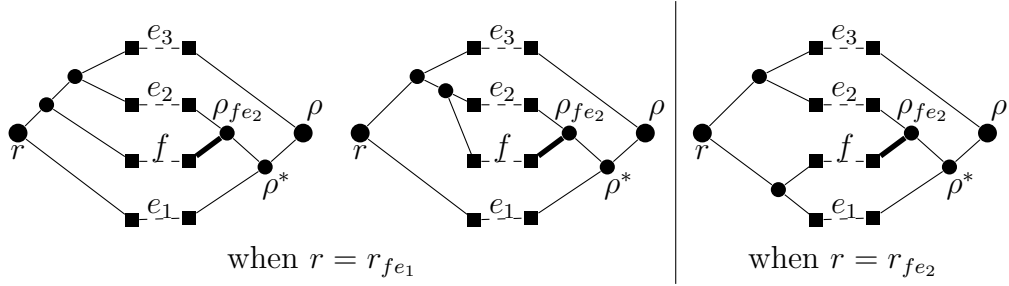

If $m$ lies on the unique $\rho^*$-$e_2$ path in $R$ (including the case that $m=m^*$), then the subtanglegram induced by $e_1,e_2,e_3$ satisfies the conclusion of our lemma and we are done.
Otherwise let $f$ be any matching edge such that $m$ lies on the unique  $\rho^*$-$f$ path in $R$. By our assumptions,
$f\notin\{e_1,e_2\}$, $f$ lies between $e_1$ and $e_2$ in our planar layout and $\rho_{fe_2}$ lies on the unique $\rho^*$-$e_2$ path in
$R$ (Figure~\ref{fig:mcases}). We have
$r\in\{r_{e_1f},r_{e_2f}\}$. If $r=r_{e_1f}$, then the subtanglegram of $F$ induced by $e_1,e_3,f$ satisfies the conclusion of our lemma.
If $r=r_{e_2f}$, then the subtanglegram induced by $e_1,e_2,f$ satisfies the conclusion. Either way, we are done.
\end{proof}

\begin{lemma}\label{lem:zip} Let $F$ be a tanglegram with two sets of matching edges,  $E_1,E_2$, such that 
$E_1\cap E_2=\{f\}$ and $E_1\cup E_2$ contains all matching edges of $F$. For $i\in\{1,2\}$, let 
$F_i$  be the subtanglegram induced by $E_i$, and assume
that the scars of the edges of $E_1$ in $F_2$ as well as the scars of the edges of $E_2$ in $F_1$ are on a unique root-to-root path containing $f$ but no other matching edge.
If $F_1$ and $F_2$ each have planar layouts in which the two matching edges on the boundary  of the infinite face are $f$ and  $e_1$ and correspondingly $f$ and $e_2$, then $F$  has a planar layout in which the matching edges on the boundary of the infinite face are $e_1$ and $e_2$.
\end{lemma}

\begin{proof} Let the left and right roots of $F$ be $r$ and $\rho$ respectively, and let $P$ be the unique $r$-$\rho$ path in $F$ containing
$f$ but no other matching edges, and let $r^1,\rho^1$ and $r^2,\rho^2$ be the left and right roots of $F_1$ and $F_2$ respectively. From the assumptions on $f$ we get that $r^1,r^2,\rho^2,\rho^2$ lie on $P$. 

The conditions on $F_1$ and $F_2$ imply that $F_1$ has a planar layout such that the $r^1$-$\rho^1$ path $P_1$ containing $f$ lies on a straight line,
 all other edges of $F_1$ lie above this line and $e_1$ is on the boundary of the infinite face; also,
 $F_2$ has a planar layout such that the $r^2$-$\rho^2$ path $P_2$ containing $f$ lies on a straight line,
 all other edges of $F_2$ lie below this line and $e_2$ is on the boundary of the infinite face. Since the order of vertices on $P$
 is independent of the drawings, $P_1$ and $P_2$ can be obtained from subpaths of $P$ by suppressing some vertices, so these two
layouts can be merged into the required planar layout of $F$; see Figure~\ref{fig:lem_zip} for an illustration of this lemma.
\end{proof}

\begin{figure}[htbp]
\begin{center}
\begin{tikzpicture}[scale=0.7,line/.style={-}]

 	\draw[gray!20, fill=gray!20] (2,-4)--(3,-4) --(3,4)-- (2,4) --(2,-4);

	\draw[dashed] (2,0)--(3,0);
	\draw[dashed] (2,-4)--(3,-4);
	\draw[dashed] (2,4)--(3,4);

	\draw[dotted] (2,1)--(3,1);
	\draw[dotted] (2,2)--(3,2);
	\draw[dotted] (2,3)--(3,3);

	\draw[dotted] (2,-1)--(3,-1);
	\draw[dotted] (2,-2)--(3,-2);
	\draw[dotted] (2,-3)--(3,-3);

         \node at (2.5,4.2) {$e_1$};
         \node at (2.5,0.2) {$f$};
         \node at (2.5,-4.3) {$e_2$};

         \node at (2.5,-2.2) {$E_2$};
         \node at (2.5,2.2) {$E_1$};

   	\path[black,line,line width=2pt] (-2,0) edge (2,0);
   	\path[black,line,line width=2pt] (3,0) edge (7,0);
	\path[black,line,line width=2pt,in=-170, out=30] (-2,0) edge (2,4);
	\path[black,line,line width=2pt,in=160, out=-70] (-1,0) edge (2,-4);
	\path[black,line,line width=2pt,in=20, out=-160] (7,0) edge (3,-4);
	\path[black,line,line width=2pt,in=-10, out=130] (6.5,0) edge (3,4);
        
\end{tikzpicture}
\end{center}
\caption{Illustration of Lemma~\ref{lem:zip}.}\label{fig:lem_zip}
\end{figure}
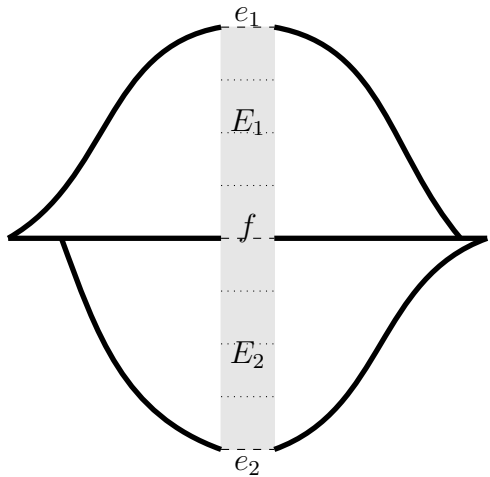

\section{Crossing-critical tanglegrams}

Another key concept in this paper is that of a {\em crossing-critical tanglegram}. A tanglegram is crossing-critical if it is non-planar, but every proper induced subtanglegram of it is planar. For example, tanglegrams No.~6 and No.~13 are crossing-critical. Clearly any non-planar
tanglegram contains a crossing-critical induced subtanglegram.

\begin{theorem} \label{thm:main}
The only crossing-critical tanglegrams are No.~6 and No.~13. Therefore, every non-planar tanglegram contains No.~6 or No.~13. as
an induced subtanglegram. 
\end{theorem}

\begin{corollary} \label{cor:t4remaining}
The remaining eleven tanglegrams of size 4 in Figure~\ref{fig:t4} are planar.
\end{corollary}
\begin{corollary}
For every non-planar tanglegram $T$, the augmented graph $T^*$ contains a subdivision of $K_{3,3}$, where three of the original vertices of the
$K_{3,3}$ are located in $L$ and the other three in $R$.
\end{corollary}

It would be interesting to see if a more general theorem holds for tanglegrams exhibiting an even higher degree of non-planarity:

\begin{question} For an integer $k\geq 3$, is there a characterization of tanglegrams that have $k$ pairwise crossing matching edges
in every layout, in terms of a finite list of tanglegrams that they must have as induced subtanglegrams, analogous to Theorem~\ref{thm:main}? 
\end{question}

\begin{proof}[Proof of Theorem~\ref{thm:main}]
Assume that $T$ is a crossing-critical tanglegram, with left subtree $L$ rooted at $r$ and right subtree $R$ rooted at $\rho$.
Let $L_u,L_d$ be the rooted subtrees of $L$ rooted at the neighbors $r^u$ and $r^d$ of $r$ and $R_u,R_d$ be the rooted subtrees of $R$ rooted at the neighbors 
$\rho^u$ and $\rho^d$ of $\rho$. 
Since the leaves of both $L_u$ and $L_d$ are matched to the leaves of at least one of $R_u$ and $R_d$, and vice versa, we may assume without loss of generality that  there are matching edges between $L_u$ and $R_u$, and between $L_d$ and $R_d$. (If  this is not the case, we can achieve 
this situation using switch operations.)

Denote the non-empty set of matching edges
between $L_u$ and $R_u$ by $E_u$, and between $L_d$ and $R_d$ by $E_d$, and let $E_m$ be the (potentially empty) set of matching edges not in $E_u\cup E_d$.

Let $T_u$ and $T_d$ be the subtanglegrams of $T$ induced by the matching edges $E_u$ and $E_d$, respectively. Since $T$ is 
crossing-critical, both $T_u$ and $T_d$ 
are planar tanglegrams.

If $E_m=\emptyset$ (part (a) in Figure~\ref{fig:triple}) then $T$ has a planar layout (just put the planar layouts of $T_u$ and $T_d$ above each other, and connect the
 vertex $r$ to the left roots of $T_u$ and $T_d$, and $\rho$ to the right roots of $T_u$ and $T_d$), which is a contradiction. Therefore $E_m\ne\emptyset$.

If $E_m$ contains a matching edge $g$ between $L_u$ and $R_d$ and a matching edge $f$ between $L_d$ and $R_u$ (part (b) in Figure~\ref{fig:triple}), then let $e\in E_u$ and
$h\in E_d$. The subtanglegram induced by the edges $e,f,g,h$ in $T$ is No. 13, and we are done.
So we are left to consider the case when only one of the pairs $L_u,R_d$ and $L_d,R_u$ has matching edges between them, in this case
without loss of generality (using the
mirror image operation at $r$ and $\rho$, if needed) $E_m$ is the non-empty set of matching edges between $L_u$ and $R_d$ (part (c) in Figure~\ref{fig:triple}).

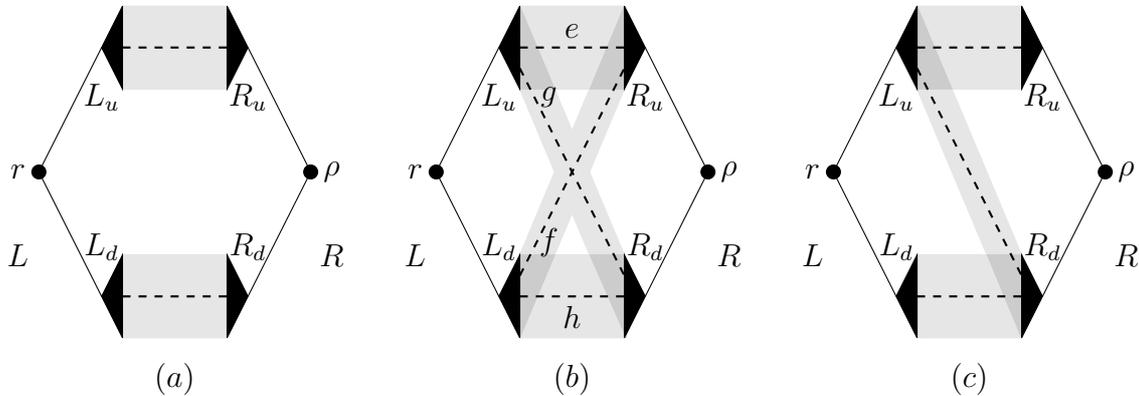
\begin{figure}[htbp]
\begin{center}
\begin{tikzpicture}[scale=0.55]
	\draw[gray!20, fill=gray!20]    (2.5,2) --(5,2) --(5,4) -- (2.5,4) --(2.5,2);	
	\draw[gray!20, fill=gray!20]    (2.5,-2) --(5,-2) --(5,-4) -- (2.5,-4) --(2.5,-2);		
        \node[fill=black,circle,inner sep=2pt]  at (0.5,0) {};
        \node[fill=black,circle,inner sep=2pt]  at (7,0) {};
	\node at (0,0) {$r$};
	\node at (7.5,0) {$\rho$};
	\node at  (2,1.8) {$L_u$};
	\node at (2,-1.8) {$L_d$};
	\node at (5.5,1.8) {$R_u$};
	\node at (5.5,-1.8) {$R_d$};
	\draw[fill=black] (2,3)--(2.5,2)--(2.5,4);
	\draw (2.5,4)--(0.5,0)--(2.5,-4);
	\draw[fill=black] (2.5,-4)--(2.5,-2)--(2,-3);
	\draw[fill=black] (5.5,3)--(5,2)--(5,4);
	\draw (5,4)--(7,0)--(5,-4);
	\draw[fill=black] (5,-4)--(5,-2)--(5.5,-3);
	\node at (0,-2) {$L$};
	\node at (7.5,-2) {$R$};
	\node at (3.75,-5) {($a$)};
	\draw[thick,dashed] (2.5,3)--(5,3);
	\draw[thick,dashed] (2.5,-3)--(5,-3);

	\draw[gray!20, fill=gray!20]    (12,2) --(14.5,2) --(14.5,4) -- (12,4) --(12,2);	
	\draw[gray!20, fill=gray!20]    (12,-2) --(14.5,-2) --(14.5,-4) -- (12,-4) --(12,-2);	
	\draw[gray!20, fill=gray!20]    (12,2) --(14.5,-4) --(14.5,-2) -- (12,4) --(12,2);	
	\draw[gray!20, fill=gray!20]    (12,-2) --(14.5,4) --(14.5,2) -- (12,-4) --(12,-2);
	\draw[gray!40,fill=gray!40]    (12,4)--(12.833,2)--(12,2);	
	\draw[gray!40,fill=gray!40]    (12,-4)--(12.833,-2)--(12,-2);	
	\draw[gray!40,fill=gray!40]    (14.5,4)--(13.677,2)--(14.5,2);	
	\draw[gray!40,fill=gray!40]    (14.55,-4)--(13.677,-2)--(14.5,-2);				
        \node[fill=black,circle,inner sep=2pt]  at (10,0) {};
        \node[fill=black,circle,inner sep=2pt]  at (16.5,0) {};
	\node at (9.5,0) {$r$};
	\node at (17,0) {$\rho$};
	\node at  (11.5,1.8) {$L_u$};
	\node at (11.5,-1.8) {$L_d$};
	\node at (15,1.8) {$R_u$};
	\node at (15,-1.8) {$R_d$};
	\node at (13.25,3.4) {$e$};
	\node at (13.25,-3.5) {$h$};
	\node at (12.7,1.8) {$g$};
	\node at (12.7,-1.7) {$f$};
	\draw[fill=black] (11.5,3)--(12,2)--(12,4);
	\draw (12,4)--(10,0)--(12,-4);
	\draw[fill=black] (12,-4)--(12,-2)--(11.5,-3);
	\draw[fill=black] (15,3)--(14.5,2)--(14.5,4);
	\draw (14.5,4)--(16.5,0)--(14.5,-4);
	\draw[fill=black] (14.5,-4)--(14.5,-2)--(15,-3);
	\draw[thick,dashed] (12,3)--(14.5,3);
	\draw[thick,dashed] (12,-3)--(14.5,-3);
	\draw[thick,dashed] (12,2.5)--(14.5,-2.5);
	\draw[thick,dashed] (12,-2.5)--(14.5,2.5);
	\node at (9.5,-2) {$L$};
	\node at (17,-2) {$R$};
	\node at (13.25,-5) {($b$)};

	\draw[gray!20, fill=gray!20]    (21.5,2) --(24,2) --(24,4) -- (21.5,4) --(21.5,2);	
	\draw[gray!20, fill=gray!20]    (21.5,-2) --(24,-2) --(24,-4) -- (21.5,-4) --(21.5,-2);	
	\draw[gray!20, fill=gray!20]    (21.5,2) --(24,-4) --(24,-2) -- (21.5,4) --(21.5,2);	
	\draw[gray!40,fill=gray!40]    (21.5,4)--(22.333,2)--(21.5,2);	
	\draw[gray!40,fill=gray!40]    (24,-4)--(23.177,-2)--(24,-2);				
        \node[fill=black,circle,inner sep=2pt]  at (19.5,0) {};
        \node[fill=black,circle,inner sep=2pt]  at (26,0) {};
	\node at (19,0) {$r$};
	\node at (26.5,0) {$\rho$};
	\node at  (21,1.8) {$L_u$};
	\node at (21,-1.8) {$L_d$};
	\node at (24.5,1.8) {$R_u$};
	\node at (24.5,-1.8) {$R_d$};
	\draw[fill=black] (21,3)--(21.5,2)--(21.5,4);
	\draw (21.5,4)--(19.5,0)--(21.5,-4);
	\draw[fill=black] (21.5,-4)--(21.5,-2)--(21,-3);
	\draw[fill=black] (24.5,3)--(24,2)--(24,4);
	\draw (24,4)--(26,0)--(24,-4);
	\draw[fill=black] (24,-4)--(24,-2)--(24.5,-3);
	\node at (19,-2) {$L$};
	\node at (26.5,-2) {$R$};
	\node at (22.75,-5) {($c$)};
	\draw[thick,dashed] (21.5,3)--(24,3);
	\draw[thick,dashed] (21.5,-3)--(24,-3);
	\draw[thick,dashed] (21.5,2.5)--(24,-2.5);

\end{tikzpicture}
\end{center}
\caption{Case analysis on the qualitative distribution of matching edges between subtrees of $L$ and $R$. Dashed lines mark the existence of matching edges between the subtrees.}\label{fig:triple}
\end{figure}

We will consider two cases:

Case (A): $\min(|E_u|,|E_d|)\ge 2$. We are going to show that this does not happen in a crossing-critical tanglegram $T$.\\
Let $e\in E_d$ and consider the subtanglegram $T'$ induced by all matching edges except $e$ with left tree $L'$ and right tree $R'$. 
$T'$ is planar, contains $T_u$ as a subtanglegram, and contains the vertices $\rho^u$ and $\rho^d$. As the unique path from the root to a matching edge in $R'$ passes through $\rho^u$ for every matching edge in $E_u$ and passes through $\rho^d$ for every matching edge in $E_d\cup E_m$, in any planar layout of $T'$
the edges of $E_u$ appear contiguously, and the edges of $E_m$ appear  on only one side of them.
Consequently, any planar layout of $T'$ gives a planar sublayout of $T_u$ where
all scars from $E_m$ lie on the same root-to-root path bordering the infinite face; denote the matching edge
which this path travels through by $f_u$. Similar logic gives that
$T_d$ has a planar layout in which all scars lie on the same root-to-root path bordering the infinite face, denote the
matching edge which this path travels through by $f_d$. Let $T''$ be the tanglegram induced by $E_m\cup\{f_u,f_d\}$. 
Consider a planar layout of $T''$ (as $T$ is crossing-critical, such a layout exists), without loss of generality
(up to a mirror operation) $f_u$ lies above $f_d$ in this layout. Let $P_u$ be the unique shortest path leading from $r$ to
$f_u$ and $P_d$ be the unique shortest path leading from $\rho$ to $f_d$, and let $g\in E_m$ be arbitrary.
Consider the vertical strip between the two vertical lines going though $r$ and $\rho$ -- this is the region where $T''$ is drawn.
The $r$-$\rho$ paths containing $f_u$ and $f_d$ and no other matching edge
cut this strip into three subregions,
and $g$ must lie in the unique subregion that borders both $P_u$ and $P_d$. That means $g$ lies between 
$f_u$ and $f_d$ in this planar layout of $T''$, and consequently
 in any planar layout of $T''$, $f_u$ and $f_d$ are on the boundary of the infinite face.
Two applications of Lemma~\ref{lem:zip} (first on $T''$ and $T_u$ using the common edge $f_u$, then on the resulting 
tanglegram and $T_d$ using the common edge $f_d$)
show that $T$ itself has a planar layout, which is a contradiction. 

\begin{figure}[htbp]
\begin{center}
\begin{tikzpicture}[line/.style={-}]
	\draw[gray!20,fill=gray!20] (-2,2.5)--(0,2.5)--(0,-2.3)--(-2,-2.3)--(-2,2.5);
	\draw[gray!20,fill=gray!20] (9,2.5)--(7,2.5)--(7,-2.3)--(9,-2.3)--(9,2.5);	
        \node(A)[fill=black,circle,inner sep=2pt]  at (0,0) {};
        \node(E)[fill=black,circle,inner sep=2pt]  at (7,0) {};
        \node(B)[fill=black,rectangle, inner sep=2pt] at (3,2) {};
        \node(F)[fill=black, rectangle, inner sep=2pt] at (4,2) {};
        \node(C)[fill=black,rectangle, inner sep=2pt] at (3,-2) {};
        \node(G)[fill=black, rectangle, inner sep=2pt] at (4,-2) {};
        \node(D)[fill=black,circle,inner sep=2pt]  at (1.4,.7) {};
        \node(H)[fill=black, rectangle, inner sep=2pt] at (6.2,-1.1) {};        
        \draw[dashed] (3,2)--(4,2);
        \draw[dashed] (3,-2)--(4,-2);
        \path[line width=2pt,in=130,out=30] (A) edge (D);      
        \path[line width=2pt,in=-170,out=30] (D) edge (B);
        \path[line width=2pt,in=70,out=-100] (E) edge (H);      
        \path[line width=2pt,in=70,out=-100] (H) edge (G);
       \path[in=100,out=-70] (A) edge (C);
       \path[in=-700,out=100] (E) edge (F);
      
        \draw[thick,dotted] (0,2.3)--(0,-2.3);      
        \draw[thick,dotted] (7,2.3)--(7,-2.3);  
        \node at (-.3,0) {$r$};
        \node at (7.3,0) {$\rho$};  
        \node at (1.4,.4) {$r_{g,f_u}$};  
        \node at (5.7,.-1.1) {$\rho_{g,f_d}$};  
        \node at (3.5,2.3) {$f_u$};  
        \node at (3.5,-1.7) {$f_d$};  
        \node at (1.1,1.2) {$P_u$};  
        \node at (5.5,-1.9) {$P_d$};  
\end{tikzpicture}
\end{center}
\caption{Analysis of a planar drawing of the subtanglegram $T''$. The white region between the dotted vertical
lines is where $T''$ is drawn.}\label{fig:caseA}
\end{figure}
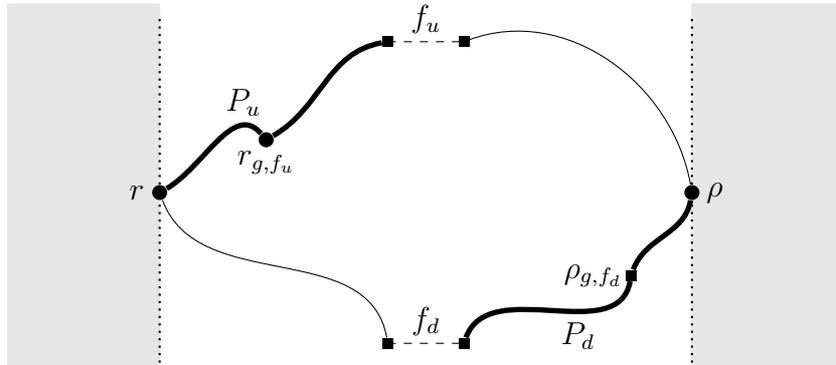

Case (B): $\min(|E_u|,|E_d|)=1$. We are going to exhibit No. 6 as an induced subtanglegram in $T$.\\
Assume first that $|E_d|=1$, and let $e$ be the single matching edge between $L_d$ and $R_d$.
This means in particular that $L_d$ consists of a  single leaf vertex $r^d$ that is matched by the edge $e$.
Now, by the crossing-criticality of $T$, the subtanglegram induced by all matching edges but $e$, denoted by
$\hat T$, is planar, and a non-matching  edge of  its right subtree, $\hat R$, has a scar marking $e$ (this scar exists, as $R_d$ has leaves
matched by $E_m$ and therefore $\rho_{\sigma\setminus\{e\}}=\rho$).  
If $\hat{T}$ has a planar layout in which the marked edge is on the boundary of the infinite face, then $T$ has a planar layout, contradicting the crossing-criticality of $T$. Therefore
in all planar layouts of $\hat{T}$, the marked edge is not on the boundary of the infinite face. Lemma~\ref{lem:kukac} shows that $\hat{T}$ contains the subtanglegram $S$ with the mark $m$ positioned as in Figure~\ref{fig:megvan},
and, using the fact  that $r$ is connected to one of the endpoints of $e$, we find  that the subtanglegram induced by the edges
$a,b,c,e$ is No. 6, so we are done. If $|E_u|=1$, the argument is essentially the same after exchanging the roles of $L$ and $R$.
\end{proof} 

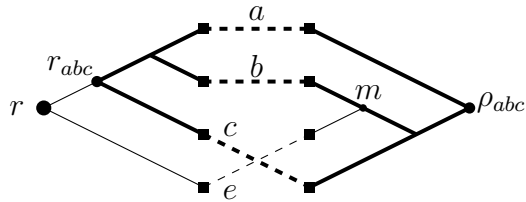
\begin{figure}[htbp]
\begin{center}
\begin{tikzpicture}[scale=.7]
\node at (0,0) {$r$};
\node[fill=black,circle,inner sep=2pt]  at (0.5,0) {};
\node[fill=black,circle,inner sep=1.5pt] at (1.5,.5) {};
\draw (1.5,0.5)--(0.5,0)--(3.5,-1.5);
\draw[line width=1.5pt] (3.5,1.5)--(1.5,0.5)--(3.5,-0.5); 
\draw[line width=1.5pt] (2.5,1)--(3.5,0.5);
\draw[dashed,line width=1.5pt] (3.5,1.5)--(5.5,1.5);
\draw[dashed,line width=1.5pt] (3.5,-0.5)--(5.5,-1.5);
\draw[dashed] (3.5,-1.5)--(5.5,-.5);
\draw[dashed,line width=1.5pt] (3.5,.5)--(5.5,.5);
\node at (9.1,0.1) {$\rho_{abc}$};
\node at (1,.8) {$r_{abc}$};
\draw[line width=1.5pt] (5.5,1.5)--(8.5,0)--(5.5,-1.5);
\draw[line width=1.5pt] (5.5,.5)--(7.5,-.5);
\draw (5.5,-.5)--(6.5,0);
\node at (6.6,0.3) {$m$};
\node[fill=black,circle,inner sep=1pt]  at (6.5,0) {};
\node[fill=black,rectangle,inner sep=2pt]  at (3.5,1.5) {};
\node[fill=black,rectangle,inner sep=2pt]  at (3.5,.5) {};
\node[fill=black,rectangle,inner sep=2pt]  at (3.5,-.5) {};
\node[fill=black,rectangle,inner sep=2pt]  at (3.5,-1.5) {};
\node[fill=black,rectangle,inner sep=2pt]  at (5.5,1.5) {};
\node[fill=black,rectangle,inner sep=2pt]  at (5.5,.5) {};
\node[fill=black,rectangle,inner sep=2pt]  at (5.5,-.5) {};
\node[fill=black,rectangle,inner sep=2pt]  at (5.5,-1.5) {};
\node at (4.5,1.8) {$a$};
\node at (4.5,.8) {$b$};
\node at (4,-.4) {$c$};
\node at (4,-1.6) {$e$};
\node[fill=black,circle,inner sep=1.5pt]  at (8.5,0) {};

\end{tikzpicture}
\end{center}
\caption{Finding subtanglegram No. 6 in $T$ starting from a copy of $S$ drawn in bold.}\label{fig:megvan}
\end{figure}

%\bibliographystyle{abbrv}
%\bibliography{Inducibility}

\end{document}